\documentclass[a4paper,12pt]{article}
\usepackage{latexsym}
\usepackage{a4wide}
\usepackage{amscd}
\usepackage{graphics}
\usepackage{graphicx}
\usepackage[english]{varioref}
\usepackage{amsmath}
\usepackage{amssymb}
\usepackage{mathrsfs}
\usepackage{amsthm}
\usepackage{amssymb,tikz}
\usepackage{verbatim}
\usepackage{stmaryrd}
\usepackage[xindy]{glossaries}

\usepackage{tikz}
\tikzstyle{mybox} = [draw=black, very thick, rectangle, rounded corners, inner ysep=5pt, inner xsep=5pt]

\usepackage{bbm}
\usepackage{soul}
\usepackage{soul,color}
\usepackage{accents}
\usepackage{enumerate}
\usepackage[utf8x]{inputenc}
\usepackage[T1]{fontenc}
\usepackage[english,francais]{babel}
\usepackage{amsfonts, amscd, amsmath, amssymb, amsthm, mathrsfs, mathtools}
\usepackage{braket}
\input xy
\xyoption{all}
\usepackage[utf8x]{inputenc}
\usepackage[T1]{fontenc}
\usepackage[english]{babel}
\usepackage{amsfonts, amscd, amsmath, amssymb, amsthm, mathrsfs, mathtools}
\usepackage{graphicx}
\usepackage{subfig}
\usepackage{braket}
\usepackage[left=2.5cm,right=2.5cm,top=2.5cm,bottom=2.5cm]{geometry}
\usepackage{fancyhdr}
\theoremstyle{definition}

\newtheorem{remark}{Remark}[section]

\newtheorem{esempio}{Example}[section]
\newtheorem{notazione}{Notation}[section]

\newtheorem{definizione}{Definition}[section]
\newtheorem{fatto}{Fact}[section]
\newtheorem{proprieta}{Property}[section]

\theoremstyle{plain}
\newtheorem{teorema}{Theorem}[section]
\newtheorem{proposizione}{Proposition}[section]
\newtheorem{lemma}{Lemma}[section]
\newtheorem{corollario}{Corollary}[section]

\newcommand{\numberset}{\mathbb}

\newcommand{\R}{\numberset{R}}
\newcommand{\N}{\numberset{N}}
\newcommand{\Z}{\numberset{Z}}
\newcommand{\A}{\mathbb{A}}
\newcommand{\T}{\mathbb{T}}

\usepackage{bookmark,hyperref}

\DeclarePairedDelimiter{\abs}{\lvert}{\rvert}
\DeclarePairedDelimiter{\norma}{\lVert}{\rVert}

\makeatletter
\let\oldabs\abs
\def\abs{\@ifstar{\oldabs}{\oldabs*}}
\let\oldnorma\norma
\def\norma{\@ifstar{\oldnorma}{\oldnorma*}}

\title{On the set of points of zero torsion for negative-torsion maps of the annulus}

\author{\hspace{-1cm}{\qquad\ \  Anna Florio${\,}^{1}$}\vspace{0.3cm}
	\\
	\hspace{-1cm}${\ }^{1}$ Fondation Sciences Mathématiques de Paris, CNRS ,\\
	\hspace{-1cm}IMJ-PRG, 4 Place Jussieu, F-75005, Paris, France\\}

\date{\today}

\begin{document}

\selectlanguage{english}
\maketitle
\begin{abstract}
	\noindent For negative-torsion maps on the annulus we show that on every $\mathcal{C}^1$ essential curve there is at least one point of zero torsion. As an outcome we deduce that the Hausdorff dimension of the set of points of zero torsion is greater or equal 1. As a byproduct we obtain a Birkhoff's-theorem-like result for $\mathcal{C}^1$ essential curves in the framework of negative-torsion maps.
\end{abstract}
\section{Introduction}
\indent Let $\A=\T\times\R$ where $\T=\R/\Z$. Endow $\A$ with the standard Riemannian metric and the standard trivialization. Let $f:\A\rightarrow\A$ be a $\mathcal{C}^1$ diffeomorphism isotopic to the identity. Let $(f_t)_{t\in[0,1]}$ be an isotopy in $\text{Diff }^1(\A)$ joining the identity to $f_1=f$. We are interested in the linearized dynamics: more precisely we look at the \textit{torsion}, a dynamical invariant first introduced by D.~Ruelle in 1985 (see \cite{ruelle}), who called it \emph{rotation number}.\\
\noindent Roughly speaking, the torsion of the orbit of a point $x$, denoted as $\text{Torsion}(f,x)$, describes the average asymptotic velocity at which the differential of the diffeomorphism makes the tangent vectors turning along the considered orbit. The torsion at finite time $T$ of a point $(x,\xi)$ of the tangent bundle is the variation between $0$ and $T$ of a continuous determination of the angle function associated to $Df_t(x)\xi,t\in[0,T]$, divided by $T$. The torsion of the orbit of $x$ is the limit for $T\rightarrow+\infty$ of the torsion at finite time $T$, whenever the limit exists. If $\mu$ is a compact-supported $f$-invariant Borel probability measure , then the torsion of the orbit of $x$ exists for $\mu$-almost every point $x\in\A$, see \cite{ruelle}.\\
\noindent Besides Ruelle's work, in the framework of conservative twist maps, the structure of some null torsion sets, called Aubry-Mather sets, has been studied by Mather (in \cite{matherexistence} and \cite{mathervariational}) and Angenent (in \cite{Ang}) through a variational approach. Using topological tools, Crovisier has generalised some results in the non conservative twist framework. He has shown that for every rotation number there exists an Aubry-Mather sets of zero torsion, see Theorem 1.2 in \cite{crovisier}.\\
\indent A negative-torsion map is a $\mathcal{C}^1$ diffeomorphism isotopic to the identity such that at every point the vertical vector has a negative average rotational velocity: that is, for every point $x\in\A$ the torsion at finite time $1$ at $x$ with respect to the vertical vector is negative. Positive twist maps and Tonelli Hamiltonian flows at finite time are examples of negative-torsion maps. The notion of negative-torsion map is equivalent to the notion of positive tilt maps as presented in \cite{Hu} and \cite{GidRob}.\\
\noindent In this paper we consider the set of points of zero torsion for negative-torsion maps. The main result of the work is the following.
\begin{teorema}\label{teo 1}
	Let $f:\A\rightarrow\A$ be a negative-torsion map. Then for any $\mathcal{C}^1$ essential curve $\gamma:\T\rightarrow\A$ there exists a point $z\in\gamma(\T)$ such that $\text{Torsion}(f,z)=0$.
\end{teorema}
\noindent Remark that we do not ask any conservative hypothesis on $f$. Applying Theorem \ref{teo 1} to every simple circle curves, i.e. any curve $\T\times\{r\}$ for $r\in\R$, we can deduce the following
\begin{corollario}\label{cor 1}
	Let $f:\A\rightarrow\A$ be a negative-torsion map. Then
	$$
	\text{dim}_H\left( \{ z\in\A :\ \text{Torsion}(f,z)=0 \} \right)\geq 1,
	$$
	where $\text{dim}_H(\cdot)$ denotes the Hausdorff dimension of the set.
\end{corollario}
\noindent The idea of the proof of Theorem \ref{teo 1} is considering, for every $N\in\N^*$, the preimage on $\gamma$ of points of maximal height of $f^N\circ\gamma(\T)$. We then show that the angle variation of the vector tangent to $\gamma$ at these points between 0 and $N$ is bounded uniformly in $N$. By the negative-torsion property, we control the torsion at finite time $m$ of these points for every $m\in\llbracket 1,n\rrbracket$. Consider then the sequence of preimages on $\gamma$ of such points of maximal height. Every limit point of the sequence is a point of zero torsion.\\
\noindent The idea of looking at points of maximal height can be adapted to extend the result to $\mathcal{C}^0$ essential curves which are graphs of functions (see Theorem \ref{teo c0 graph}).\\
~\newline
\indent As an outcome of Theorem \ref{teo 1}, we deduce that the torsion of an orbit of a point belonging to a $\mathcal{C}^1$ $f$-invariant essential curve can be calculated through the angle variation of the vector tangent to the curve along the curve itself. We obtain so as a by-product of the proof of Theorem \ref{teo 1} a version of Birkhoff's theorem for negative-torsion maps.
\begin{teorema}\label{teo 2}
	Let $f:\A\rightarrow\A$ be a negative-torsion map. Let $\gamma:\T\rightarrow\A$ be a $\mathcal{C}^1$ $f$-invariant essential curve such that $f_{\vert \gamma}$ is non wandering. Then $\gamma$ is the graph of a $\mathcal{C}^1$ function.
\end{teorema}
\indent The paper is organised as follows. In Section \ref{sect 2} we fix the notation and we provide the main definitions of torsion and negative-torsion maps. Useful properties of torsion are recalled. Section \ref{sect 3} is devoted to the discussion of the proof of Theorem \ref{teo 1}. Its demonstration relies on two main propositions, see Propositions \ref{prop building sequence} and \ref{prop bounded before}. In particular, Proposition \ref{prop torsion} is discussed first in a simpler case in Subsection \ref{subsec 1} and then in the general framework in Subsection \ref{subsection 2}. Subsection \ref{subsection3} concerns the partial generalisation of Theorem \ref{teo 1} to continuous essential curves. Finally, we present the Birkhoff's-theorem-like result (see Theorem \ref{teo 2}) for negative-torsion maps in Section \ref{Birkhoff result}.\\
~\newline

\noindent \textbf{Acknowledgements.} The author is extremely grateful to Marie-Claude Arnaud and Andrea Venturelli for all their preciuos advices and for many discussions.
\section{Notation, main definitions and first properties of torsion}\label{sect 2}
Let $\T=\R/\Z$ and let $\A=\T\times\R$. Denote as $p:\R\rightarrow\T$ and as $p\times\text{Id}:\R^2\rightarrow\A$ the universal coverings of $\T$ and of $\A$ respectively. The functions $p_1,p_2$ denote the projections over the first and the second coordinate respectively over $\A$ and, with an abuse of notation, also over $\R^2$.\\
\noindent Endow $\A$ with the standard Riemannian metric and the standard trivialization. We fix the counterclockwise orientation of $\R^2$. Thus the notion of oriented angle between two non zero vectors is well-defined.
\begin{definizione}[see \cite{Hirsch}]
	Let $M,N$ be differential manifolds and let $f,g:M\rightarrow N$ be in $\text{Diff }^1(M,N)$. An isotopy $(\phi_t)_{t\in[0,1]}$ joining $f$ to $g$ is an arc in $\text{Diff }^1(M,N)$ such that $\phi_0=f,\phi_1=g$ and which is continuous with respect to the weak or compact-open $\mathcal{C}^1$ topology on $\text{Diff }^1(M,N)$.
\end{definizione}
\begin{definizione}
	Let $I\subset \R$ be an interval. A continuous determination of an angle function $\theta:I\rightarrow\T$ is a continuous lift of $\theta$, that is a continuous function $\tilde{\theta}:I\rightarrow\R$ such that $\tilde{\theta}(s)$ is a measure of the oriented angle $\theta(s)$ for any $s\in I$.
\end{definizione}
Let $f:\A\rightarrow\A$ be a $\mathcal{C}^1$ diffeomorphism isotopic to the identity. Let $(f_t)_{t\in[0,1]}$ be an isotopy joining the identity to $f$. Extend the isotopy for any positive time so that for any $t\in\R_+$ the $\mathcal{C}^1$ diffeomorphism $f_t:\A\rightarrow\A$ is defined as $f_t:=f_{\{t\}}\circ f^{\lfloor t\rfloor}$, where $\{\cdot\},\lfloor \cdot \rfloor$ are the fractionary and integer part respectively. We denote $I=(f_t)_{t\in\R_+}$ the extended isotopy. For any $x\in\A$ denote as $\chi=(0,1)\in T_x\A$ the unitary positive vertical vector.\\
\noindent The definition of torsion we adopt is the one given by Béguin and Boubaker in \cite{beguin}. The asymptotic torsion is actually Ruelle's rotation number, see \cite{ruelle}.
\begin{definizione}
	Let $x\in\A,\xi\in T_x\A,\xi\neq 0$. Define the oriented angle function
	\begin{equation}
	\R_+\ni t\mapsto v(I,x,\xi)(t):=\theta(\chi,Df_t(x)\xi)\in\T,
	\end{equation}
	where $\theta(u,v)$ denotes the oriented angle between the two non zero vectors $u,v$.\\
	\noindent Denote
	\begin{equation}
	\R_+\ni t\mapsto \tilde{v}(I,x,\xi)(t)\in\R
	\end{equation}
	a continuous determination of the continuous oriented angle function $v(I,x,\xi)(\cdot)$.\\
	\noindent The torsion at finite time $n\in\N^*$ of $(x,\xi)\in T\A,\xi\neq 0$ is
	\begin{equation}
	\text{Torsion}_n(I,x,\xi):=\dfrac{\tilde{v}(I,x,\xi)(n)-\tilde{v}(I,x,\xi)(0)}{n}.
	\end{equation}
	The torsion at $x\in\A$ is, whenever it exists,
	\begin{equation}
	\text{Torsion}(I,x):=\lim_{n\rightarrow+\infty}\text{Torsion}_n(I,x,\xi).
	\end{equation}
\end{definizione}
\begin{remark}\label{prop torsion}
	The torsion at finite time does not depend on the choice of the continuous determination of the oriented angle function. Moreover, it is independent from the choice of the isotopy joining the identity to $f$, see Proposition 2.5 in \cite{Flo19}. The (asymptotic) torsion, whenever it exists, does not depend on the tangent vector used to calculate the finite time torsion. We refer to \cite{beguin} for these properties.
\end{remark}
\begin{notazione}
	Since the torsion does not depend on the chosen isotopy, we write\\ \noindent $\text{Torsion}_n(f,x,\xi)$ and $\text{Torsion}(f,x)$.
\end{notazione}
\begin{esempio}\label{es pend}
	Consider the dynamical system of the simple pendulum obtained by the Hamiltonian $H(\theta,r)=\frac{r^2}{2}-\frac{\cos(2\pi \theta)}{4\pi^2}$. Let $(\phi_t)_{t\in\R}$ be the associated flow and consider the time-one flow $f=\phi_1$. Let $U$ denote the open region contained between the separatrices of the pendulum system, see Figure \ref{figpend}.\\
	\noindent Every point $z$ not belonging to $U$ has zero (asymptotic) torsion. The elliptic point $(0,0)$ has torsion equal to $-1$. Every point $z\in U\setminus\{(0,0)\}$ is periodic and has torsion equal to $-\frac{1}{T(z)}$, where $T(z)$ is the period of $z$.
\end{esempio}
\begin{figure}[h]
	\centering
	\includegraphics[scale=.1]{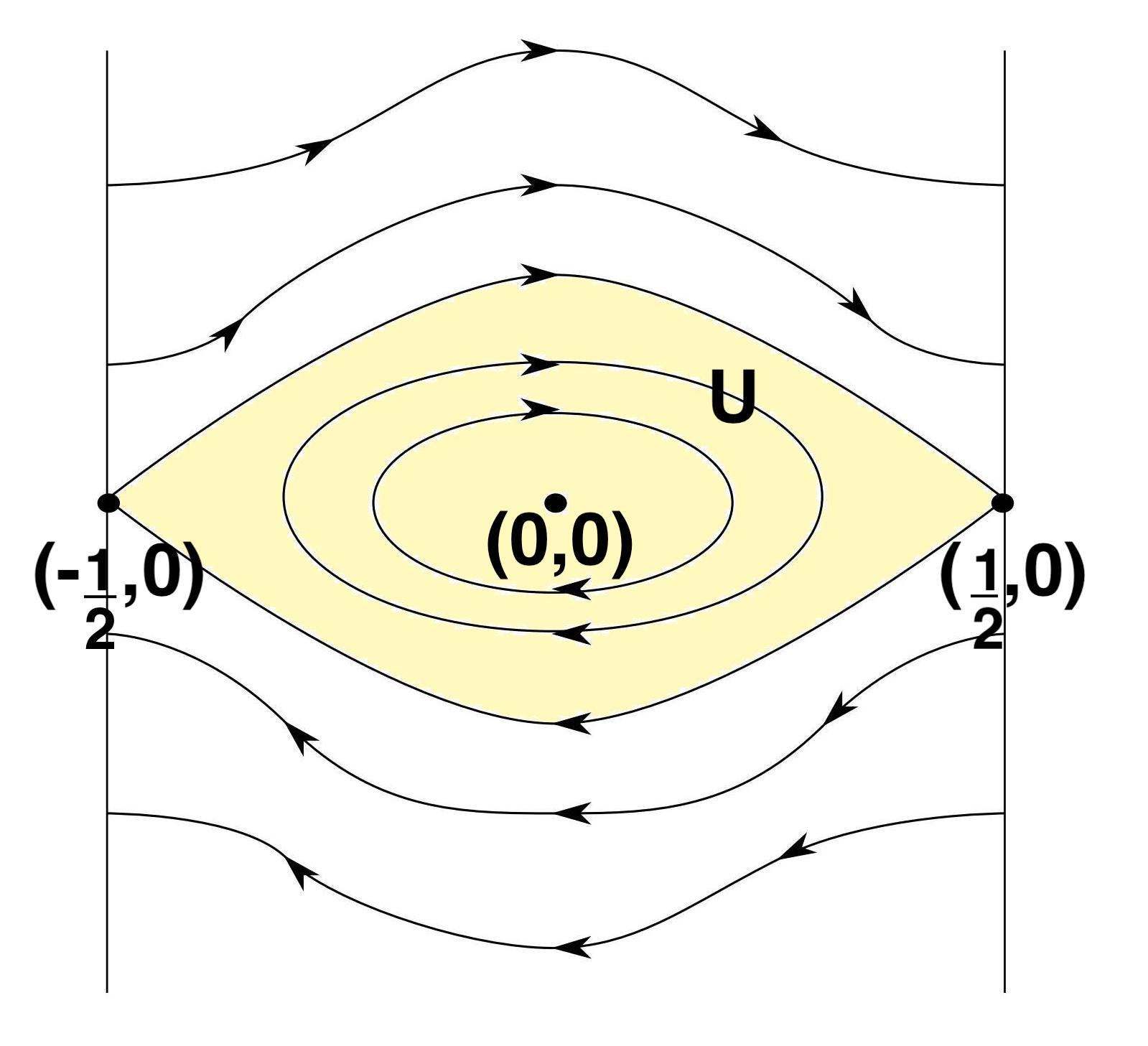}
	\caption{Phase portrait of the pendulum system of Example \ref{es pend}.}
	\label{figpend}
\end{figure}
\begin{definizione}
	A $\mathcal{C}^1$ diffeomorphism isotopic to the identity $f:\A\rightarrow\A$ is a negative-torsion map (respectively a positive-torsion map) if for any $z\in\A$ it holds
	\begin{equation}
	\text{Torsion}_1(f,z,\chi)<0\quad\left(\text{ respectively }>0\right).
	\end{equation}
\end{definizione}
\begin{esempio}
	Every positive twist map is a negative-torsion map (see \cite{Flo19}). Actually the notion of negative-torsion map is equivalent to the notion of positive tilt map: see \cite{Hu}, \cite{GidRob} and \cite{FloPHD}. We remark that the notion of negative-torsion (positive-torsion) maps can be given also in terms of positive (negative) paths according to the definitions in \cite{Her83} and \cite{PLCB}.
\end{esempio}
\begin{definizione}
	An essential curve is a $\mathcal{C}^0$ embedded circle in $\A$ not homotopic to a point.
\end{definizione}

\noindent We recall here some properties of finite-time torsion. We refer to \cite{beguin}, \cite{Flo19} and \cite{FloPHD} for the proofs.
\begin{proprieta}\label{indep vector}
	Fix $x\in\A$. If we calculate finite-time torsion at $x$ with respect to different tangent vectors, we can control the error. Indeed, for $n\in\N^*$ and for $\xi,\delta\in T_x\A\setminus\{0\}$
	$$
	\abs{\text{Torsion}_n(f,x,\xi)-\text{Torsion}_n(f,x,\delta)}<\dfrac{1}{2n},
	$$
	see Lemma 2.1 in \cite{Flo19}.
\end{proprieta}
\begin{proprieta}\label{order vector angle}
Let $f:\A\rightarrow\A$ be a $\mathcal{C}^1$ diffeomorphism isotopic to the identity. Let $I$ be an isotopy in $\text{Diff }^1(\A)$ joining the identity to $f$. Let $x\in\A$ and $\xi_1,\xi_2\in T_x\A\setminus\{0\}$. Let $\tilde{v}(I,x,\xi_1)(\cdot),\tilde{v}(I,x,\xi_2)(\cdot)$ be continuous determinations of the angle functions $v(I,x,\xi_1)(\cdot),v(I,x,\xi_2)(\cdot)$ respectively.\\
\noindent If $$\tilde{v}(I,x,\xi_1)(0)>\tilde{v}(I,x,\xi_2)(0),$$ then for any $t\in\R$ it holds $$\tilde{v}(I,x,\xi_1)(t)>\tilde{v}(I,x,\xi_2)(t).$$ See Proposition 2.2 in \cite{Flo19}.
\end{proprieta}
\begin{proprieta}\label{lemma techinque}
	Let $f:\A\rightarrow\A$ be a $\mathcal{C}^1$ diffeomorphism isotopic to the identity. Let $a\in\A$. Let $\mathscr{N}\in\N^*,(\mathscr{K}_i)_{i\in\llbracket 0,\mathscr{N}-1\rrbracket}\in\N^{\mathscr{N}}$ and $l_0=0<l_1<\dots<l_{\mathscr{N}}$ with $l_i\in\N$. Assume that for all $i\in\llbracket 0,\mathscr{N}-1\rrbracket$ it holds
	$$
	(l_{i+1}-l_i)\text{Torsion}_{l_{i+1}-l_i}(f,f^{l_i}(a),\chi)<-\dfrac{\mathscr{K}_i}{2}.
	$$
	Then for any vector $\xi\in T_a\A\setminus\{0\}$ we have
	$$
	l_{\mathscr{N}}\text{Torsion}_{l_{\mathscr{N}}}(f,a,\xi)<-\dfrac{\sum_{i=0}^{\mathscr{N}-1}\mathscr{K}_i}{2}+\dfrac{1}{2}.
	$$
	In particular, when $\xi=\chi$, we have
	$$
	l_{\mathscr{N}}\text{Torsion}_{l_{\mathscr{N}}}(f,a,\chi)<-\dfrac{\sum_{i=0}^{\mathscr{N}-1}\mathscr{K}_i}{2}.
	$$
	We refer to \cite{FloPHD}.
\end{proprieta}
\noindent Thus, from Property \ref{lemma techinque}, we bound from above finite-time torsion for negative-torsion maps.
\begin{proprieta}\label{order neg tors}
	Let $f:\A\rightarrow\A$ be a negative-torsion map. Let $m\in\N^*$ and let $z\in\A$ be such that $m\text{Torsion}_m(f,z,\chi)<-\frac{k}{2}$ for some $k\in\N^*$. Then for any $n\geq m,n\in\N^*$ it holds $n\text{Torsion}_n(f,z,\chi)<-\frac{k}{2}$.
\end{proprieta}
\section{Set of points of zero torsion}\label{sect 3}
\noindent Theorem \ref{teo 1} is an outcome of the following two propositions.
\begin{proposizione}\label{prop building sequence}
	Let $f:\A\rightarrow\A$ be a $\mathcal{C}^1$ diffeomorphism isotopic to the identity. Let $\gamma:\T\rightarrow\A$ be a $\mathcal{C}^1$ essential curve. There exists $C\in\R_+$ such that for any $n\in\N^*$ there exists $s_n\in\T$ so that
	\begin{equation}
	\abs{n\text{Torsion}_n(f,\gamma(s_n),\gamma'(s_n))}\leq C.
	\end{equation}
\end{proposizione}
\noindent We postpone the proof of Proposition \ref{prop building sequence} to Subsections \ref{subsec 1} and \ref{subsection 2}. In Subsection \ref{subsec 1}, we will first show it in the simpler case of simple circle curves, that is $\gamma(\T)=\T\times\{r\}$ for $r\in\R$, using the link between torsion and linking number. Then we will prove the proposition in the general case of $\mathcal{C}^1$ essential curves, see Subsection \ref{subsection 2}.
\begin{proposizione}\label{prop bounded before}
	Let $f:\A\rightarrow\A$ be a negative-torsion map. Let $C\in\R_+$ and let $z\in\A$ be such that $\abs{n\text{Torsion}_n(f,z,\chi)}\leq C$ for some $n\in\N^*$. Let $K=\lfloor 2C\rfloor+2$. Then for any $m\in\llbracket 1, n\rrbracket$ it holds
	\begin{equation}
	m\text{Torsion}_m(f,z,\chi)\in\Big[-\dfrac{K}{2},0\Big).
	\end{equation}
\end{proposizione}
\begin{proof}
	Since $f$ is a negative-torsion map and by Property \ref{order neg tors}, for any $z\in\A$ and  for any $m\in\N^*$ it holds $m\text{Torsion}_m(f,z,\chi)<0$. Let $z\in\A$ and $n\in\N$ be such that $\abs{n\text{Torsion}_n(f,z,\chi)}\leq C$. In particular, by the negative-torsion property, it holds $$n\text{Torsion}_n(f,z,\chi)\in[-C,0).$$ Argue by contradiction and assume that there exists $m\in\llbracket 1,n\rrbracket$ such that
	$$
	m\text{Torsion}_m(f,z,\chi)<-\dfrac{K}{2}\leq -C-\dfrac{1}{2}.
	$$
	If $m=n$ we contradict the hypothesis. Thus, we have $m<n$. Again because $f$ is a negative-torsion map, it holds $(n-m)\text{Torsion}_{n-m}(f,f^m(z),\chi)<0$.\\
	\noindent Apply then Property \ref{lemma techinque} for $f$ at $z$ with respect to $\mathscr{N}=2,l_1=m,l_2=n,\mathscr{K}_1=K,\mathscr{K}_2=0$. We so obtain
	$$
	n\text{Torsion}_n(f,z,\chi)<-\dfrac{K}{2}<-C,
	$$
	which is the required contradiction.
\end{proof}

\noindent \textit{Proof of Theorem \ref{teo 1}.} Let $(s_n)_{n\in\N^*}\subset\T^{\N^*}$ be the sequence of points built in Proposition \ref{prop building sequence}. That is, for any $n\in\N^*$ it holds
$$
\abs{n\text{Torsion}_n(f,\gamma(s_n),\gamma'(s_n))}\leq C,
$$
where $C$ does not depend on $n$. By the properties of finite-time torsion, see Property \ref{indep vector}, we have
$$
\abs{n\text{Torsion}_n(f,\gamma(s_n),\chi)}\leq C+\dfrac{1}{2}.
$$
Denote as $s_{\infty}\in\T$ a limit point of the sequence $(s_n)_{n\in\N^*}$. This is our candidate point of zero torsion.\\
\noindent Fix $N\in\N^*$. Let $\varepsilon>0$. Up to subsequences and by the continuity of finite time torsion with respect to the point, there exists $\bar{n}\in\N^*,\bar{n}>N$ such that
$$
\abs{N\text{Torsion}_N(f,\gamma(s_{\infty}),\chi)-N\text{Torsion}_N(f,\gamma(s_{\bar{n}}),\chi)}<\varepsilon.
$$
Let $K=\lfloor 2C+1\rfloor+2$. By Proposition \ref{prop bounded before} and since $f$ is a negative-torsion map (see Property \ref{order neg tors}), it holds
$$
0> N\text{Torsion}_N(f,\gamma(s_{\infty}),\chi)=
$$
$$
=\left(N\text{Torsion}_N(f,\gamma(s_{\infty}),\chi)-N\text{Torsion}_N(f,\gamma(s_{\bar{n}}),\chi)\right)+N\text{Torsion}_N(f,\gamma(s_{\bar{n}}),\chi)>-\varepsilon-\dfrac{K}{2}.
$$
Thus
$$
\text{Torsion}_N(f,\gamma(s_{\infty}),\chi)\in\Big[-\dfrac{K}{2N},0\Big).
$$
In particular $K$ is independent from $N$ and, as $N$ goes to $+\infty$, we conclude that $$\text{Torsion}(f,\gamma(s_{\infty}))=0.$$
\hfill\qed
\subsection{A first simpler case: simple circle curves and Corollary \ref{cor 1}}\label{subsec 1}
\noindent Proposition \ref{prop building sequence} can be proved easily in the framework of simple circle curves, that is if $\gamma(\T)=\T\times\{r\}$ for $r\in\R$, by using the notion of linking number and the relation between linking number and torsion. The linking number for a diffeomorphism of the plane of two points measures the average rotational velocity at which the orbit of the first point turns around the orbit of the second one. We refer to \cite{beguin} for a deeper discussion.

\begin{definizione}
Let $I=(F_t)_{t\in\R_+}$ be an isotopy in $\text{Diff }^1(\R^2)$ joining the identity to $F_1=F$ and such that $F_{1+t}=F_t\circ F$. Fix $x,y\in\R^2,x\neq y$. Denote $\chi=(0,1)$. Define the oriented angle function
$$
\R_+\ni t\mapsto u(I,x,y)(t):=\theta(\chi,F_t(y)-F_t(x))\in\T,
$$
where $\theta(u,v)$ denotes the oriented angle between the two non zero vectors $u,v$.\\
\noindent Since $u(I,x,y)(\cdot)$ is continuous, consider a continuous determination $\R_+\ni t\mapsto \tilde{u}(I,x,y)(t)\in\R$ of such oriented angle function.\\
\noindent For every $n\in\N^*$ the linking number at finite time $n$ of $x$ and $y$ is
$$
\text{Linking}_n(I,x,y):=\dfrac{\tilde{u}(I,x,y)(n)-\tilde{u}(I,x,y)(0)}{n}.
$$
\noindent The linking number of $x$ and $y$ is, whenever the limit exists,
$$
\text{Linking}(I,x,y):=\lim_{n\rightarrow+\infty}\text{Linking}_n(I,x,y).
$$
\end{definizione}
\noindent Concerning the relation between torsion and linking number, we recall here Corollary 3.1 in \cite{Flo19}. The torsion is calculated with respect to the standard trivialization. The vertical vector $\chi=(0,1)$ is used to measure oriented angles.
\begin{corollario}[Corollary 3.1 in \cite{Flo19}]\label{cor flo 3.1}
	Let $F:\R^2\rightarrow\R^2$ be a $\mathcal{C}^1$ diffeomorphism isotopic to the identity and let $I$ be an isotopy joining the identity to $F_1=F$. Assume there exist $n\in\N^*$ and $x,y\in\R^2,x\neq y$ such that $\text{Linking}_n(I,x,y)=l\in\R$. Then there exists $z\in[x,y]$, where $[x,y]$ denotes the segment joining the points $x,y$, such that $\text{Torsion}_n(I,z,y-x)=l$.
\end{corollario}
\noindent We proceed now with the proof of Proposition \ref{prop building sequence} in the case of a simple circle curve.
\begin{proposizione}
	Let $f:\A\rightarrow\A$ be a $\mathcal{C}^1$ diffeomorphism isotopic to the identity. Let $r\in\R$ and let $n\in\N^*$. Then there exists $z(r,n)\in\T\times\{r\}$ such that
	\begin{equation}
	\text{Torsion}_n(f,z(r,n),\mathcal{H})=0,
	\end{equation}
	where $\mathcal{H}=(1,0)\in T_{z(r,n)}\A$ is the horizontal positive unitary vector.
\end{proposizione}
\begin{proof}
	Let $F:\R^2\rightarrow\R^2$ be a lift of $f$. Let $I=(F_t)_{t\in\R_+}\in\text{Diff }^1(\R^2)$ be the isotopy joining the identity of $\R^2$ to $F$, obtained as lift of an isotopy on $\A$ joining $\text{Id}_{\A}$ to $f$.\\
	\noindent Observe that for any $t\in\R_+$ the function $F_t$ commutes with the translation by $(1,0)$. Consequently, for any fixed $r\in\R$ and for any $n\in\N^*$ it holds
	$$
	\text{Linking}_n(I,(0,r),(1,r))=0.
	$$
	By Corollary 3.1 in \cite{Flo19} (here Corollary \ref{cor flo 3.1}), there exists $z(r,n)\in (p\times \text{Id})([0,1]\times\{r\})$ such that $$\text{Torsion}_n(f,z(r,n),\mathcal{H})=0,$$ where $\mathcal{H}$ is the unitary positive horizontal vector.
\end{proof}

\indent In particular, this proves Theorem \ref{teo 1} for an essential curve $\gamma$ such that $\gamma(\T)=\T\times\{r\}$ for $r\in\R$. Consequently, we can deduce Corollary \ref{cor 1}.\\

\noindent \textit{Proof of Corollary \ref{cor 1}.} By Theorem \ref{teo 1} applied at all simple circle curves, for any $r\in\R$ there exists $z(r)\in\T\times\{r\}$ such that $\text{Torsion}(f,z(r))=0$. Thus
$$
p_2\left(\{z\in\A :\ \text{Torsion}(f,z)=0\}\right)=\R.
$$
We are now interested in the Hausdorff dimension, denoted as $\text{dim}_H$, of the set of points of zero torsion. Recall that if $g$ is a Lipschitz function, then for any set $U$ it holds $\text{dim}_H(U)\geq \text{dim}_H(g(U))$.\\
\noindent Since the projection over the second coordinate $p_2$ is Lipschitz and since $\text{dim}_H(\R)=1$, we conclude that $\text{dim}_H(\{z\in\A :\ \text{Torsion}(f,z)=0\})\geq 1$.\\
\hfill\qed
\subsection{The general case of $\mathcal{C}^1$ essential curves}\label{subsection 2}

\noindent We are now interested in the general case of $\gamma$ being a $\mathcal{C}^1$ essential curve. In order to deal with this case we need to introduce the notion of angle variation along the curve $\gamma$.\\
\noindent Let $\gamma:\T\rightarrow\A$ be a $\mathcal{C}^1$ essential curve and let $x,y\in\gamma(\T)$. Let $s_1,s_2\in\T$ be such that $\gamma(s_1)=x,\gamma(s_2)=y$. Fix $S_1\in\R$ a lift of $s_1$ and let $S_2\in\R$ be the lift of $s_2$ such that $S_2\in(S_1,S_1+1]$.\\
\noindent Define the oriented angle function
$$
\R_+\ni t\mapsto \Theta(\gamma,S_1)(t):=\theta\left(\chi,\dfrac{d\gamma(\tau)}{d\tau}_{ \Big\vert\tau=p(S_1+t)}\right)\in\T,
$$
where $p:\R\rightarrow\T$ is the covering map of $\T$. Equivalently, $\Theta(\gamma,S_1)(t)$ is the oriented angle between $\chi$ and the vector tangent to $\gamma$ at $\gamma(p(S_1+t))$. Denote as $\tilde{\Theta}(\gamma,S_1):\R_+\rightarrow\R$ a continuous determination of such oriented angle function.
\begin{definizione}
The angle variation along $\gamma$ between $x=\gamma(s_1)$ and $y=\gamma(s_2)$ is $$Var_{\gamma}(x,y):=\tilde{\Theta}(\gamma,S_1)(S_2-S_1)-\tilde{\Theta}(\gamma,S_1)(0),$$
where $S_1$ is a lift of $s_1$ and $S_2$ is the lift of $s_2$ in $(S_1,S_+1]$.
\end{definizione}
\noindent Observe that the angle variation along $\gamma$ does not depend on the chosen continuous determination $\tilde{\Theta}(\gamma,S_1)$.

\begin{fatto}\label{prop var lungo gamma}
	We recall here some useful properties of the angle variation along a $\mathcal{C}^1$ essential curve $\gamma$. Let $x,y,z\in\gamma(\T)$.\vspace{5pt}
	
	\begin{itemize}
		\item[$(1)$] $Var_{\gamma}(x,y)$ does not depend on the choice of the lift $S_1$ of $s_1\in\T$ such that $\gamma(s_1)=x$;\vspace{5pt}
		
		\item[$(2)$] $Var_{\gamma}(x,x)=0$;\vspace{5pt}
		
		\item[$(3)$] $Var_{\gamma}(x,y)+Var_{\gamma}(y,z)=Var_{\gamma}(x,z)$.
	\end{itemize}
\end{fatto}
\begin{remark}
	Fix $\gamma(s)\in\gamma(\mathbb{T})$. We observe that the function
	$
	\R_+\ni t\mapsto Var_{\gamma}(\gamma(s),\gamma(s+p(t)))\in\R
	$ is $1$-periodic.
\end{remark}
\begin{remark}
	An essential curve $\gamma$ on the annulus is isotopic to either $$\T\ni t\mapsto \mathbf{c}_1(t)=(t,0)\qquad\text{or}\qquad\T\ni t\mapsto\mathbf{c}_{-1}(t)=(-t,0).$$
\end{remark}
\begin{proposizione}\label{proposizione variazione curva punti altezza max}
	Let $\gamma:\mathbb{T}\rightarrow\mathbb{A}$ be a $\mathcal{C}^1$ essential curve. Let $s_0,s_1\in\mathbb{T},s_0\neq s_1$ correspond to points of maximal height on $\gamma$, that is
	$$
p_2\circ\gamma(s_0)=p_2\circ\gamma(s_1)=\max_{s\in\mathbb{T}}p_2\circ\gamma(s),
	$$
	where $p_2:\A\rightarrow\R$ is the projection over the second coordinate. Then
	$$
	Var_{\gamma}(\gamma(s_0),\gamma(s_1))=0.
	$$
\end{proposizione}
\begin{proof}
	Let $s_0,s_1\in\mathbb{T},s_0\neq s_1$ be such that $p_2\circ\gamma(s_0)=p_2\circ\gamma(s_1)=\max_{s\in\mathbb{T}}p_2\circ\gamma(s)$. Let $S_0\in\R$ be a lift of $s_0\in\mathbb{T}$ and let $S_1\in(S_0,S_0+1)$ be the lift of $s_1$.\\
	\noindent Look now at the lifted framework in $\R^2$ and denote as $\Gamma:\R\rightarrow\R^2$ a lift of $\gamma$. Consider the points $\Gamma(S_0),\Gamma(S_1)$ and build the piecewise $\mathcal{C}^1$ closed curve $\mathscr{C}$ by concatenating the following ones (see Figure \ref{figure birk 1}):\vspace{5pt}
	
	\begin{figure}[h]
		\centering
		\includegraphics[scale=.08]{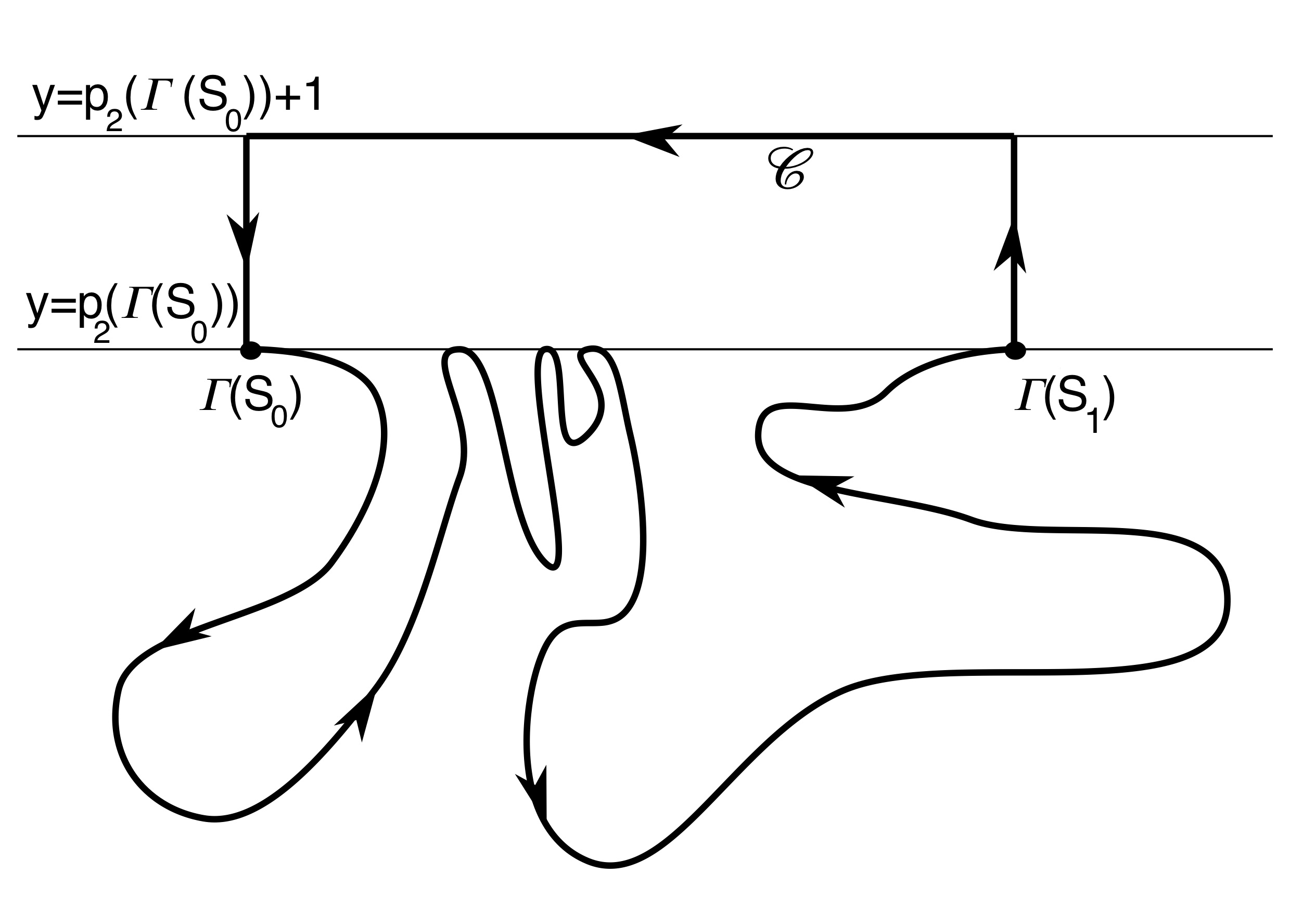}
		\caption{The simple curve built in the proof of Proposition \ref{proposizione variazione curva punti altezza max}.}
		\label{figure birk 1}
	\end{figure}
	\begin{itemize}
		\item $\{ \Gamma(s) :\ s\in[S_0,S_1] \}$;\vspace{7pt}
		
		\item the vertical segment $\{ (p_1\circ\Gamma(S_0),p_2\circ\Gamma(S_0)+\xi) :\ \xi\in[0,1] \}$;\vspace{7pt}
		
		\item the horizontal segment $\{ (\xi p_1\circ\Gamma(S_0)+(1-\xi)p_1\circ\Gamma(S_1),p_2\circ\gamma(S_0)+1) :\ \xi\in[0,1] \}$;\vspace{7pt}
		
		\item the vertical segment $\{ (p_1\circ\Gamma(S_1),p_2\circ\Gamma(S_1)+1-\xi) :\ \xi\in[0,1] \}$.\\
	\end{itemize}
	Such a piecewise $\mathcal{C}^1$ closed curve does not have self-intersections because both $\Gamma(S_0)$ and $\Gamma(S_1)$ are points of maximal height. 
	We are then interested in
	$$
	Var_{\gamma}(\gamma(s_0),\gamma(s_1))=Var_{\Gamma}(\Gamma(S_0),\Gamma(S_1))=\tilde{\Theta}(\Gamma,S_0)(S_1-S_0)-\tilde{\Theta}(\Gamma,S_0)(0).
	$$
	\noindent If $\gamma$ is homotopic to $\mathbf{c}_1$ (respectively to $\mathbf{c}_{-1}$), then \begin{equation}\label{eq claim cc}
	p_1\circ\Gamma(S_0)<p_1\circ\Gamma(S_1)\qquad\left(\text{ resp. }p_1\circ\Gamma(S_0)>p_1\circ\Gamma(S_1)\right)\end{equation}
	and
	\begin{equation}\label{eq calim cc2}
	\Gamma'(S_0),\Gamma'(S_1)\in\R_+\mathcal{H}\qquad\left(\text{ resp. }\Gamma'(S_0),\Gamma'(S_1)\in\R_-\mathcal{H}.\right).\end{equation}

	\noindent In particular, if $\gamma$ is homotopic to $\mathbf{c}_1$ (respectively to $\mathbf{c}_{-1}$) then the closed curve $\mathscr{C}$ is oriented counterclockwisely (respectively clockwisely).\\
	\noindent Apply then the Turning Tangent Theorem to the simple piecewise $\mathcal{C}^1$ closed curve $\mathscr{C}$ described above (see Figure \ref{figure birk 1}).\\
	\noindent By \eqref{eq claim cc} and \eqref{eq calim cc2}, if $\gamma$ is homotopic to $\mathbf{c}_1$, then we have
	$$
	\tilde{\Theta}(\Gamma,S_0)(S_1-S_0)-\tilde{\Theta}(\Gamma,S_0)(0) +\dfrac{1}{4}+\dfrac{1}{4}+\dfrac{1}{4}+\dfrac{1}{4}=1.
	$$
	\noindent A similar result can be obtained if $\gamma$ is homotopic to $\mathbf{c}_{-1}$. Thus
	$$
	\tilde{\Theta}(\Gamma,S_0)(S_1-S_0)-\tilde{\Theta}(\Gamma,S_0)(0)=Var_{\Gamma}(\Gamma(S_0),\Gamma(S_1))=Var_{\gamma}(\gamma(s_0),\gamma(s_1))=0.
	$$
\end{proof}
\begin{notazione}
	Let $s_0\in\T$ be a point of maximal height, that is such that $p_2\circ\gamma(s_0)=\max_{t\in\T}p_2\circ\gamma(t)$. Fix $S_0\in\R$ a lift of $s_0$.
\end{notazione}
\begin{definizione}[Complexity of a $\mathcal{C}^1$ essential curve]\label{def complexity}
The complexity of the curve $\gamma$ is
$$
C(\gamma):=\sup_{t\in\R_+}\abs{Var_{\gamma}(\gamma(p(S_0)),\gamma(p(S_0+t))}=\max_{t\in[0,1]}\abs{Var_{\gamma}(\gamma(p(S_0)),\gamma(p(S_0+t))},
$$
where $p:\R\rightarrow\T$ is the covering map of $\T$.
\end{definizione}
\begin{remark}\label{complexity curve non dipende s0}
	By Proposition \ref{proposizione variazione curva punti altezza max} and by the properties of the angle variation along $\gamma$ (see Fact \ref{prop var lungo gamma}), we remark that the definition of $C(\gamma)$ is independent of the choice of the point $s_0\in\mathbb{T}$ of maximal height.
\end{remark}

\noindent We can now prove Proposition \ref{prop building sequence} for any $\mathcal{C}^1$ essential curve $\gamma$. More precisely
\begin{proposizione}\label{prop building sequence gamma gen}
	Let $\gamma:\T\rightarrow\A$ be a $\mathcal{C}^1$ essential curve of complexity $C(\gamma)$. Let $n\in\N^*$. Then there exists $z(n)=\gamma(s_n)\in\gamma(\T)$ such that
	$$
	\abs{n\text{Torsion}_n(f,\gamma(s_n),\gamma'(s_n))}\leq C(\gamma).
	$$
\end{proposizione}
\begin{notazione}
	For any $t\in \R_+$ denote as $\gamma_t$ the curve
	$$
	\mathbb{T}\ni s\mapsto \gamma_t(s):= f_t(\gamma(s))\in\mathbb{A}.
	$$
	Consider the maximal height function
	$$
	M^h_{\gamma}:\R_+\rightarrow\R
	$$
	$$
	t\mapsto M^h_{\gamma}(t):=\max_{s\in\mathbb{T}}p_2\circ\gamma_t(s).
	$$
	\noindent For any $t\in\R_+$ denote
	\begin{equation}\label{def argmax}
	Argmax(p_2\circ\gamma_t)=\{ s\in\mathbb{T} :\ p_2\circ\gamma_t(s)=M^h_{\gamma}(t) \},
	\end{equation}
	that is $Argmax(p_2\circ\gamma_t)$ is the set of $s\in\mathbb{T}$ whose image through $\gamma_t$ achieves the maximal height among $\gamma_t(\T)$.\\
	\noindent Observe that, since each $\gamma_t$ is $\mathcal{C}^1$, for any $s\in Argmax(p_2\circ \gamma_t)$ 
	the tangent vector $\gamma_t'(s)$ belongs to $\R\mathcal{H}$. For any $t\in\R_+$ denote as $s_t$ an element of $Argmax(p_2\circ\gamma_t)$.
\end{notazione}
\begin{notazione}
	For any $t\in\R_+$ we denote as $t\,\text{Torsion}_t(I,z,\xi)$ the angle variation $\tilde{v}(I,z,\xi)(t)-\tilde{v}(I,z,\xi)(0)$.
\end{notazione}
\begin{notazione}
	Define the function $\Phi:\R_+\rightarrow \Z$
	\begin{equation}\label{definizione Phi}
	\R_+\ni t\mapsto t\,\text{Torsion}_t(f,\gamma(s_t),\gamma'(s_t))+Var_{\gamma}(\gamma(s_0),\gamma(s_t))\in \R.
	\end{equation}
	The function $\Phi$ takes values in $\Z$ because if $\gamma$ is homotopic to $\textbf{c}_1$ (respectively to $\textbf{c}_{-1}$) then both $Df_t(\gamma(s_t))\gamma'(s_t)$ and $\gamma'(s_0)$ belongs to $\R_+\mathcal{H}$ (respectively $\R_-\mathcal{H}$) (see \eqref{eq calim cc2}).
\end{notazione}

\indent The idea of considering points of maximal (respectively minimal) height on a curve is due to P.~Le Calvez (see Section 5 in \cite{PLC91}).
\begin{lemma}\label{Phi non dipende da sn per sn curva zero torsion}
	For any $t\in\R_+$, the value $\Phi(t)$ does not depend on the choice of $s_t\in Argmax(p_2\circ\gamma_t)$.
\end{lemma}

\begin{proof}
	Let $s_t,\bar{s}_t\in Argmax(p_2\circ\gamma_t), s_t\neq\bar{s}_t$. 
	From Proposition \ref{proposizione variazione curva punti altezza max} it holds that
	\begin{equation}\label{eq var uguale nei massimi}
	Var_{\gamma_t}(s_t,\bar{s}_t)=0.
	\end{equation}
	\noindent Recall that the torsion at finite-time does not depend on the chosen continuous determination of the oriented angle function. So we calculate the torsion at $\gamma(\bar{s}_t)$ using the continuous lift
	\begin{equation}\label{la correzione 30}
	\R_+\ni \tau\mapsto \tilde{v}(f,\gamma(s_t),\gamma'(s_t))(\tau)+Var_{\gamma_{\tau}}(\gamma_{\tau}(s_t),\gamma_{\tau}(\bar{s}_t))\in\R,
	\end{equation}
	where $\tilde{v}(f,\gamma(s_t),\gamma'(s_t))(\cdot)$ is a continuous lift of the angle function $\tau\mapsto\theta(\chi,Df_{\tau}(\gamma(s_t))\gamma'(s_t))$. In particular, the function in \eqref{la correzione 30} is a continuous determination of the angle function $\tau\mapsto \theta(\chi,\gamma'_{\tau}(\bar{s}_t))$. The value $\Phi(t)$ calculated with respect to $\bar{s}_t\in Argmax(p_2\circ\gamma_t)$ is
	$$
	t\,\text{Torsion}_t(f,\gamma(\bar{s}_t),\gamma'(\bar{s}_t))+ Var_{\gamma}(\gamma(s_0),\gamma(\bar{s}_t)).
	$$
	Write then $t\, \text{Torsion}_t(f,\gamma(\bar{s}_t),\gamma'(\bar{s}_t))$ using the continuous determination in \eqref{la correzione 30}. Using the properties of $Var_{\gamma_t}$ (see Fact \ref{proposizione variazione curva punti altezza max}) and by \eqref{eq var uguale nei massimi}, we can conclude that
	$$
	t\,\text{Torsion}_t(f,\gamma(\bar{s}_t),\gamma'(\bar{s}_t))+Var_{\gamma}(\gamma(s_0),\gamma(\bar{s}_t))=t\,\text{Torsion}_t(f,\gamma(s_t),\gamma'(s_t))+Var_{\gamma}(\gamma(s_0),\gamma(s_t)),
	$$
	that is $\Phi(t)$ does not depend on the choice of $s_t\in Argmax(p_2\circ\gamma_t)$.\\
\end{proof}
\begin{lemma}\label{Phi e costante nulla}
	The function $\Phi:\R_+\rightarrow \Z$ is the constant zero function.
\end{lemma}
\begin{proof}
	If we show that $\Phi$ is continuous, then, since $\Phi$ takes values in $\Z$ and since $\Phi(0)=0$, we conclude that $\Phi$ is the constant zero function.\\
	\noindent Consider the function $\Phi_{\vert[0,1]}:[0,1]\rightarrow\R$. If its graph is compact, then $\Phi_{\vert[0,1]}$ is continuous.\\
	\noindent Denote for any $t\in[0,1]$
	$$
	\mathbb{K}_t=\{ s\in\mathbb{T} :\ s\in Argmax(p_2\circ\gamma_t) \}\times\{t\}
	$$
	and
	$$
	\mathbb{K}=\bigcup_{t\in[0,1]}\mathbb{K}_t=\bigcup_{t\in[0,1]}\{ (s,t) :\ s\in Argmax(p_2\circ\gamma_t) \}\subset \mathbb{T}\times[0,1].
	$$
	\noindent The set $\mathbb{K}$ is bounded. Let $(s_n,t_n)_{n\in\N}\subset\mathbb{K}$ be a sequence converging to $(s,t)$. The sequence $(t_n)_{n\in\N}\subset [0,1]$ converges to $t\in[0,1]$. Moreover, $s\in Argmax(p_2\circ\gamma_t)$. That is, $\mathbb{K}$ is closed. Consequently, $\mathbb{K}$ is compact. Consider now the function
	$$
	\mathbb{K}\ni(s,t)\mapsto (t, t\,\text{Torsion}_t(f,\gamma(s),\gamma'(s))+Var_{\gamma}(\gamma(s_0),\gamma(s)))\in[0,1]\times\R.
	$$
	It is continuous and, since $\mathbb{K}$ is compact, its image is compact too. Observe that its image is actually the graph of the function $\Phi_{\vert[0,1]}$. Thus $\Phi_{\vert[0,1]}$ is continuous.\\
	\noindent Using the same argument, we deduce that the function $\Phi$ is continuous on every compact $[0,n]$ for $n\in\N^*$. Consequently, the function $\Phi:\R_+\rightarrow\Z$ is continuous. This concludes the proof.\\
\end{proof}
\noindent\textit{Proof of Proposition \ref{prop building sequence gamma gen}.} Fix $n\in\N^*$ and let $s_n\in Argmax(p_2\circ\gamma_n)$. By Lemma \ref{Phi non dipende da sn per sn curva zero torsion} the value $\Phi(n)$ does not depend on the element of $Argmax(p_2\circ\gamma_n)$ and by Lemma \ref{Phi e costante nulla} the function $\Phi$ is the constant zero function. Therefore
$$
\Phi(n)=n\,\text{Torsion}_n(f,\gamma(s_n),\gamma'(s_n))+Var_{\gamma}(\gamma(s_0),\gamma(s_n))=0.
$$
Thus
$$
\abs{n\,\text{Torsion}_n(f,\gamma(s_n),\gamma'(s_n))}=\abs{Var_{\gamma}(\gamma(s_0),\gamma(s_n))}\leq C(\gamma),
$$
that is $z(n):=\gamma(s_n)\in\gamma(\mathbb{T})$ is the required point.\\
\hfill\qed
\subsection{A partial result on $\mathcal{C}^0$ essential curves}\label{subsection3}
\indent It seems natural wondering if Theorem \ref{teo 1} can be generalised to continuous essential curves. We provide here a partial result.
\begin{teorema}\label{teo c0 graph}
	Let $f:\A\rightarrow\A$ be a negative-torsion map. Let $\gamma:\T\rightarrow\A$ be a $\mathcal{C}^0$ essential curve such that $\gamma(\T)$ is the graph of a function. Then there exists a point $z\in\gamma(\T)$ of zero torsion.
\end{teorema}
\noindent The result follows from Proposition \ref{prop bounded before} and from the following
\begin{proposizione}\label{prop bound c0}
	Let $f:\A\rightarrow\A$ be a $\mathcal{C}^1$ diffeomorphism isotopic to the identity. Let $\gamma:\T\rightarrow\A$ be a continuous essential curve such that $\gamma(\T)$ is a graph. Let $n\in\N^*$. Then there exists $z(n)=\gamma(s_n)\in\gamma(\T)$ such that
	\begin{equation}
	\abs{n\text{Torsion}_n(f,\gamma(s_n),\chi)}\leq \dfrac{1}{4}.
	\end{equation}
\end{proposizione}
\noindent For the proof of Proposition \ref{prop bound c0} we need to introduce the notion of \textit{tilt} angle variation.
\begin{definizione}\label{def tilt}
	Let $\psi:\R\rightarrow\A$ be a $\mathcal{C}^1$ embedded curve such that $\lim_{t\rightarrow\pm\infty}p_2\circ\gamma(t)=\pm\infty$.\\
	\noindent The angle function $tilt(\gamma)$ is defined by
	$$
	\R\ni t\mapsto tilt(\gamma)(t):=\theta(\chi,\gamma'(t))\in\T,
	$$
	where $\theta(u,v)$ denotes the oriented angle between the non-zero vectors $u,v$.\\
	\noindent Let $\widetilde{tilt}(\gamma):\R\rightarrow\R$ be the continuous determination of the angle function $tilt(\gamma)$ such that if $t\in\R$ is such that
	$$
	p_2\circ\gamma(t)>p_2\circ\gamma(s)\quad\forall s<t,
	$$
	then $\widetilde{tilt}(\gamma)(t)\in[-\frac{1}{4},\frac{1}{4}]$.
\end{definizione}
\begin{notazione}
	For every $z\in\A$ denote as
	$$
	\R\ni t\mapsto V_z(t)=(p_1(z),t)\in\A
	$$
	the vertical line passing through $z$.
\end{notazione}
\noindent We can calculate the finite-time torsion looking at the continuous determination $\widetilde{tilt}$.
\begin{proposizione}\label{lien tilt=tors}
	Let $f:\A\rightarrow\A$ be a $\mathcal{C}^1$ diffeomorphism isotopic to the identity. Then for any $z=(x,y)\in\A$ it holds
	\begin{equation}
	\text{Torsion}_1(f,z,\chi)=\widetilde{tilt}(f\circ V_z)(y).
	\end{equation}
\end{proposizione}
\begin{proof}
	Recall that the time-one torsion at $z\in\mathbb{A}$ with respect to the vertical vector $\chi$ is
	$$
	\tilde{v}(f,z,\chi)(1)-\tilde{v}(f,z,\chi)(0),
	$$
	where $t\mapsto \tilde{v}(f,z,\chi)(t)$ is a lift of the oriented angle function $t\mapsto \theta(\chi,Df_t(z)\chi)$. Consider $\widetilde{tilt}(f\circ V_z)$. Observe that both $\tilde{v}(f,z,\chi)(1)$ and $\widetilde{tilt}(f\circ V_z)(y)$ are measures of the same oriented angle $\theta(\chi,Df(z)\chi)$.\\
	\noindent The continuous function
	$$
	\mathbb{A}\ni z=(x,y)\mapsto \Psi(z):= \text{Torsion}_1(f,z,\chi)-\widetilde{tilt}(f\circ V_z)(y)\in\R
	$$
	takes value in $\Z$. Thus, it is constant.\\
	\noindent We are going to exhibit a point $z\in\mathbb{A}$ such that $\Psi(z)=0$. Thus we will conclude that for any $(x,y)\in\A$
	\begin{equation}\label{torsion 1 tilt unbounded annulus}
	\text{Torsion}_1(f,(x,y),\chi)=\widetilde{tilt}(f\circ V_{(x,0)})(y).
	\end{equation}
	\noindent Consider the $\mathcal{C}^1$ essential curve $\mathbb{T}\times\{0\}$ and its image $f(\mathbb{T}\times\{0\})$. The complexity of the curve $C(\T\times\{0\})$ is clearly zero, see Definition \ref{def complexity}.\\
	\noindent Let $\bar{z}=(x,0)\in\mathbb{T}\times\{ 0 \}$ correspond to a point of maximal height of $f(\T\times\{0\})$, that is
	$$
	p_2\circ f(\bar{z})=\max_{\xi\in\T\times\{0\}}p_2\circ f(\xi).
	$$
	\noindent By Lemmas \ref{Phi non dipende da sn per sn curva zero torsion} and \ref{Phi e costante nulla}, it holds that $\text{Torsion}_1(f,\bar{z},\mathcal{H})=0$. By Property \ref{indep vector} we have
	\begin{equation}\label{una spiegazione tilt unbounded}
	\abs{\text{Torsion}_1(f,\bar{z},\chi)}=\abs{\text{Torsion}_1(f,\bar{z},\chi)-\text{Torsion}_1(f,\bar{z},\mathcal{H})}<\dfrac{1}{2}.
	\end{equation}
	We now show that the point $\bar{z}=(x,0)\in\mathbb{T}\times\{0\}$ is such that for any $s<0$ it holds
		$$
		p_2\circ f(x,0)>p_2\circ f(x,s).$$ \noindent Indeed if by contradiction there exists $\hat{s}<0$ such that $p_2\circ f(x,\hat{s})\geq p_2\circ f(x,0)$, then, since $f(x,0)$ is a point of maximal height of $f(\T\times\{0\})$ and since $f$ preserves the boundaries, the curve $f(\T\times\{0\})$ should intersect the curve $\{f(x,\xi) :\ \xi\leq \hat{s}\}$. This contradicts the fact that
		$
		(\mathbb{T}\times\{0\})\cap \{(x,\xi) :\ \xi<0\}
		$
		is empty and that $f$ is a diffeomorphism.\\
		\noindent Consequently, from the definition of $\widetilde{tilt}$ (see Definition \ref{def tilt}) we have
		\begin{equation}\label{due spiegazioni tilt unbounded}
		\widetilde{tilt}(f\circ V_x)(0)\in\left[-\frac{1}{4},\frac{1}{4}\right].
		\end{equation}
		\noindent Look now at $\text{Torsion}_1(f,(x,0),\chi)$. Choose the continuous determination such that $\tilde{v}(f,(x,0),\chi)(0)=0$. Both $\tilde{v}(f,(x,0),\chi)(1)$ and $\widetilde{tilt}(f\circ V_x)(0)$ are measure of the same angle. Thus$$\tilde{v}(f,(x,0),\chi)(1)-\widetilde{tilt}(f\circ V_x)(0)=\text{Torsion}_1(f,(x,0),\chi)-\widetilde{tilt}(f\circ V_x)(0)\in \Z.$$
		\noindent From \eqref{una spiegazione tilt unbounded} and from \eqref{due spiegazioni tilt unbounded}, we have that
		$$
		\text{Torsion}_1(f,(x,0),\chi)=\widetilde{tilt}(f\circ V_x)(0),
		$$
		concluding so the proof.\\
\end{proof}

\noindent \textit{Proof of Proposition \ref{prop bound c0}.} Fix $n\in\N^*$. Consider $f^n\circ\gamma(\T)$ and let $s_n\in\T$ be such that
$$
p_2\circ f^n\circ\gamma(s_n)=\max_{s\in\T}p_2\circ f^n\circ\gamma(s).
$$
\noindent Denote as $V_{\gamma(s_n)}$ the vertical line passing through $\gamma(s_n)$. Denote $\gamma(s_n)=V_{\gamma(s_n)}(y_n)$. Since $\gamma(\T)$ is a graph, $V_{\gamma(s_n)}$ intersects $\gamma(\T)$ only once.\\
\noindent Consequently for every $y<y_n$ it holds
$$
p_2\circ f^n\circ V_{\gamma(s_n)}(y)<p_2\circ f^n\circ V_{\gamma(s_n)}(y_n)=p_2\circ f^n\circ\gamma(s_n)
$$
because $f^n\circ\gamma(s_n)$ is a point of maximal height of $f^n\circ\gamma(\T)$ and because $f^n$ preserves the boundaries (otherwise we would have another point of intersection between $\gamma(\T)$ and $V_{\gamma(s_n)}(\R)$).\\
\noindent By the definition of the continuous determination $\widetilde{tilt}(V_{\gamma(s_n)})$, see Definition \ref{def tilt}, we have $\widetilde{tilt}(V_{\gamma(s_n)})(y_n)\in[-\frac{1}{4},\frac{1}{4}]$.\\
\noindent By Proposition \ref{lien tilt=tors} we conclude that
$$
\text{Torsion}_1(f^n,\gamma(s_n),\chi)=n\text{Torsion}_n(f,\gamma(s_n),\chi)\in\left[-\dfrac{1}{4},\dfrac{1}{4}\right].
$$
\hfill\qed

\section{A Birkhoff's theorem through torsion}\label{Birkhoff result}

\indent Using the tool of torsion, we can prove a Birkhoff's-theorem-like result (see \cite{Birk22} and \cite{Her83}) in a different hypothesis framework, see Theorem \ref{teo 2}. The idea of using the torsion in order to prove a Birkhoff's-theorem-like result was already present in the work of M.~Bialy and L.~Polterovich (see \cite{BiaPol89}, \cite{Pol} and \cite{BiaPol}). This result arises from a question by V.~Humiliére.\\
\noindent On one hand we do not require that $f$ is either a twist map or a conservative map. On the other hand $f$ has to be a negative-torsion (positive-torsion) map and we require that the dynamics restricted to the $\mathcal{C}^1$ curve is non-wandering (see \cite{katok}).
\begin{lemma}\label{lemma graph cond}
	Let $\gamma$ be a $\mathcal{C}^1$ essential curve. If $\gamma$ is transversal to the vertical at every point, then $\gamma$ is the graph of a function.
\end{lemma}
\begin{proof}
	Let $\Gamma:\R\rightarrow\R^2$ be a lift of $\gamma$. Consider the $\mathcal{C}^1$ function $p_1\circ\Gamma:\R\rightarrow\R$. Since $\Gamma$ is transversal to the vertical at every point, it holds $D(p_1\circ\Gamma)(t)\neq 0$ for every $t\in\R$. Without loss of generality assume that it is positive at every $t\in\R$. Thus, $p_1\circ\Gamma$ is an increasing diffeomorphism to its image.\\
	\noindent Since $\Gamma$ is a lift of an essential curve and since $p_1\circ\Gamma$ is increasing, we have that for every $t\in\R$
	$$
	p_1\circ\Gamma(t+1)=p_1\circ\Gamma(t)+1.
	$$
	In particular, we deduce that $p_1\circ\Gamma(\R)=\R$. That is $p_1\circ\Gamma$ is a $\mathcal{C}^1$ diffeomorphism. Denote $\phi=(p_1\circ\Gamma)^{-1}$. Consequently, the $\mathcal{C}^1$ function $\R\ni s\mapsto p_2\circ\Gamma\circ\phi(s)\in\R$ is such that its graph is $\Gamma(\R)$.\\
	\noindent The function $p_2\circ\Gamma\circ\phi$ is $1$-periodic. Thus, its projection on the annulus is well-defined and $\gamma(\T)$ is the graph of the $\mathcal{C}^1$ function $\psi:\T\rightarrow\R$ such that $\psi\circ p=p_2\circ\Gamma\circ\phi$, where $p:\R\rightarrow\T$ is the covering map of $\T$.
\end{proof}
\noindent We will now provide an upper bound of the $N$-finite time torsion along the curve $\gamma$. The bound is independent from $N$.
\begin{notazione}
	Let $x\in\A$ and let $\delta\in(0,\frac{1}{4})$. Denote
	$$
	C(x,\chi,\delta):=\left\{ v\in T_x\A :\ \theta(\chi,v)\text{ or }\theta(-\chi,v)\text{ admits a measure in }(-\delta,\delta) \right\}.
	$$
\end{notazione}
\begin{lemma}\label{lemma bound torsion}
Let $f:\A\rightarrow\A$ be a negative-torsion map and let $K$ be a compact $f$-invariant set. There exist $\varepsilon\in(0,\frac{1}{2})$ and $\delta\in(0,\frac{\varepsilon}{4})$ such that for any $x\in K$, for any $v\in C(x,\chi,\delta)$ and for any $n\in\N^*$ it holds
\begin{equation}
N\text{Torsion}_N(f,x,v)<-\dfrac{\varepsilon}{2}<0.
\end{equation}
\end{lemma}
\begin{proof}
	Let us argue by induction. Since $f$ is a negative-torsion map, since $K$ is compact and by the continuity of the torsion at time 1, there exist $\varepsilon\in(0,\frac{1}{2})$ and $\delta\in(0,\frac{\varepsilon}{4})$ such that for every $x\in K$ and for every $v\in C(x,\chi,\delta)$ it holds
	\begin{equation}
	\text{Torsion}_1(f,x,v)<-\varepsilon<0.
	\end{equation}
	Assume now that the result holds for $N-1$. Let $x\in K$ and let $v\in C(x,\chi,\delta)$. Without loss of generality assume that the oriented angle $\theta(\chi,v)$ admits a measure in $(-\delta,\delta)$. The case of $\theta(-\chi,v)$ admitting a measure in $(-\delta,\delta)$ can be discussed similarly. Choose a continuous determination of the angle so that $\tilde{v}(I,x,v)(0)\in(-\delta,\delta)$. By inductive hypothesis it holds $$\tilde{v}(I,x,v)(N-1)<-\frac{\varepsilon}{2}+\delta<-\dfrac{\varepsilon}{4}.$$
	Consider now the continuous determination such that $\tilde{v}(I,f^{N-1}(x),\chi)(0)=0$. We point out the fact that we are considering a continuous determination with respect to a different point in $T_K\A$. In particular
	$$
	\tilde{v}(I,x,v)(N-1)<\tilde{v}(I,f^{N-1}(x),\chi)(0).
	$$
	From Property \ref{order vector angle}, from the choice of the continuous determinations and by the base case, we have that
	$$
	\tilde{v}(I,x,v)(N)<\tilde{v}(I,f^{N-1}(x),\chi)(1)=\tilde{v}(I,f^{N-1}(x),\chi)(1)-\tilde{v}(I,f^{N-1}(x),\chi)(0)<-\varepsilon.
	$$
	Consequently, by the choice of $\tilde{v}(I,x,v)(\cdot)$, we conclude that
	$$
	N\text{Torsion}_N(f,x,v)=\tilde{v}(I,x,v)(N)-\tilde{v}(I,x,v)(0)<-\varepsilon+\delta<-\dfrac{\varepsilon}{2}.
	$$
\end{proof}
\begin{lemma}\label{lemma tors=var}
Let $f:\A\rightarrow\A$ be a negative-torsion (positive-torsion) map. Let $\gamma:\T\rightarrow\A$ be a $\mathcal{C}^1$ $f$-invariant essential curve. Then for every $s\in\T$ and for every $N\in\N^*$ it holds
\begin{equation}
N\text{Torsion}_N(f,\gamma(s),\gamma'(s))=Var_{\gamma}(\gamma(s),\gamma(s_N)),
\end{equation}
where $f^N\circ\gamma(s)=\gamma(s_N)$.
\end{lemma}
\begin{proof}
We start observing that for every $s\in\T$ both $\text{Torsion}_1(f,\gamma(s),\gamma'(s))$ and $Var_{\gamma}(\gamma(s),\gamma(s_1))$ are measures of the same oriented angle, where $\gamma(s_1)=f\circ\gamma(s)$. In particular there exists $k\in\Z$ such that for any $s\in\T$
\begin{equation}\label{eq tors=var}
\text{Torsion}_1(f,\gamma(s),\gamma'(s))=Var_{\gamma}(\gamma(s),\gamma(s_1))+k.
\end{equation}
\noindent The integer $k\in\Z$ does not depend on $s\in\T$.\\
\noindent By Theorem \ref{teo 1} there exists a point $\gamma(s_{\infty})\in\gamma(\T)$ such that $\text{Torsion}(f,\gamma(s_{\infty}))=0$.\\
\noindent At the same time, since $\gamma$ is $f$-invariant and from \eqref{eq tors=var}, we have that for any $N\in\N^*$
$$
N\text{Torsion}_N(f,\gamma(s_{\infty}),\gamma'(s_{\infty}))=Nk+\sum_{i=0}^{N-1}Var_{\gamma}(\gamma(s_i),\gamma(s_{i+1})=Nk+Var_{\gamma}(\gamma(s_{\infty}),\gamma(s_N)),
$$
where for every $i\in\llbracket 0,N\rrbracket$ we denote as $s_i\in\T$ the point such that $\gamma(s_i)=f^i\circ\gamma(s_{\infty})$.\\
\noindent Since $\gamma(s_{\infty})$ has zero torsion and since we have $\abs{Var_{\gamma}(\gamma(s_{\infty}),\gamma(s_N))}\leq C(\gamma)<+\infty$, we conclude that $k=0$.\\
\noindent In particular for every $s\in\T$ and every $N\in\N^*$ it holds
$$
N\text{Torsion}_N(f,\gamma(s),\gamma'(s))=Var_{\gamma}(\gamma(s),\gamma(s_N)),
$$
where $\gamma(s_N)=f^N\circ\gamma(s)$.
\end{proof}
\noindent An outcome of Lemma \ref{lemma tors=var} is the following corollary, already proved by S.~Crovisier for twist maps in \cite{crovisier}.
\begin{corollario}
	Let $f:\A\rightarrow\A$ be a negative-torsion (positive-torsion) map. Let $\gamma:\T\rightarrow\A$ be a $\mathcal{C}^1$ essential $f$-invariant curve on $\A$. Then, for any $s\in\T$ it holds $\text{Torsion}(f,\gamma(s))=0$.
\end{corollario}

\noindent We can now finally prove Theorem \ref{teo 2}.\vspace{7pt}

\noindent \textit{Proof of Theorem \ref{teo 2}.} Argue by contradiction and assume that $\gamma$ is not a graph. Then from Lemma \ref{lemma graph cond} there exists a point $z=\gamma(s)$ such that $\gamma'(s)\in\R\chi$. Denote
$$
\chi'=\begin{cases}
\chi\qquad\text{if }\gamma'(s)\in\R_+\chi, \\
\\
-\chi\qquad\text{if }\gamma'(s)\in\R_-\chi.
\end{cases}
$$
\noindent Let $\varepsilon\in(0,\frac{1}{2})$ and $\delta\in(0,\frac{\varepsilon}{2})$ be the parameters of Lemma \ref{lemma bound torsion} applied at the $f$-invariant compact set $\gamma(\T)$.\\
\noindent Let $U\subset \T$ be a neighborhood of $s$ such that for any $t\in U$ the oriented angle $\theta(\chi',\gamma'(t))$ admits a measure in $(-\delta,\delta)$. The dynamics $f_{\vert\gamma}$ is non-wandering and therefore there exists $N\in\N$ and $\tau\in U$ such that $\tau_N\in U$ where $f^N\circ\gamma(\tau)=\gamma(\tau_N)$.\\
\noindent From Lemma \ref{lemma tors=var} it holds $N\text{Torsion}_N(f,\gamma(\tau),\gamma'(\tau))=Var_{\gamma}(\gamma(\tau),\gamma(\tau_N))$. Observe that $Var_{\gamma}(\gamma(\tau),\gamma(\tau_N))\in(-2\delta,2\delta)$.\\
\noindent Consequently, since $\delta\in(0,\frac{\varepsilon}{4})$, we conclude that
$$
N\text{Torsion}_N(f,\gamma(\tau),\gamma'(\tau))\in\left(-\dfrac{\varepsilon}{2},\dfrac{\varepsilon}{2}\right).
$$
This contradicts Lemma \ref{lemma bound torsion} and we conclude.\\
\hfill\qed

\begin{remark}
	We have shown that the curve $\gamma$ is the graph of a function and it is always transverse to the vertical. Thus, since $\gamma$ is $\mathcal{C}^1$, we deduce that $\gamma$ is the graph of a $\mathcal{C}^1$ function.
\end{remark}
\begin{remark}
	In order to obtain the result of Theorem \ref{teo 2} we need information over the dynamics on the curve. Indeed, there exist non conservative positive twist maps that admit $\mathcal{C}^1$ essential $f$-invariant curves which are not graphs of function. See Proposition 15.3 in \cite{PLCB}.
\end{remark}

\bibliographystyle{alpha}
\bibliography{Bibliography}
\end{document}